\documentclass[12pt]{amsart}
\usepackage{enumerate, amsmath, amsthm, amsfonts, amssymb, mathrsfs, graphicx, paralist, stmaryrd}
\usepackage[all]{xy}
\usepackage[usenames, dvipsnames]{color}
\usepackage[margin=1in]{geometry} 
\usepackage[bookmarks, bookmarksdepth=2, colorlinks=true, linkcolor=blue, citecolor=blue, urlcolor=blue]{hyperref}

\newcommand{\bC}{\mathbf{C}}

\newcommand{\cE}{\mathcal{E}}

\newcommand{\bF}{\mathbf{F}}

\newcommand{\bG}{\mathbf{G}}
\newcommand{\cG}{\mathcal{G}}

\newcommand{\cH}{\mathcal{H}}

\newcommand{\rH}{\mathrm{H}}

\newcommand{\cO}{\mathcal{O}}

\newcommand{\bQ}{\mathbf{Q}}

\newcommand{\cS}{\mathcal{S}}

\newcommand{\cT}{\mathcal{T}}

\newcommand{\cV}{\mathcal{V}}

\newcommand{\bZ}{\mathbf{Z}}
\newcommand{\cZ}{\mathcal{Z}}

\let\wt\widetilde
\let\ol\overline

\renewcommand{\phi}{\varphi}

\newcommand{\arxiv}[1]{\href{http://arxiv.org/abs/#1}{{\tt arXiv:#1}}}

\makeatletter
\def\Ddots{\mathinner{\mkern1mu\raise\p@
\vbox{\kern7\p@\hbox{.}}\mkern2mu
\raise4\p@\hbox{.}\mkern2mu\raise7\p@\hbox{.}\mkern1mu}}
\makeatother

\DeclareMathOperator{\im}{im}

\DeclareMathOperator{\End}{End}

\DeclareMathOperator{\ad}{ad}
\DeclareMathOperator{\Spec}{Spec}

\DeclareMathOperator{\Hom}{Hom}

\newcommand{\GL}{\mathbf{GL}}

\makeatletter
\@addtoreset{equation}{section}
\makeatother

\numberwithin{equation}{section}
\newtheorem{theorem}[equation]{Theorem}

\newtheorem{proposition}[equation]{Proposition}
\newtheorem{lemma}[equation]{Lemma}
\newtheorem{corollary}[equation]{Corollary}

\theoremstyle{definition}
\newtheorem{rmk}[equation]{Remark}
\newenvironment{remark}[1][]{\begin{rmk}[#1] \pushQED{\qed}}{\popQED \end{rmk}}
\newtheorem{eg}[equation]{Example}

\newtheorem{defn}[equation]{Definition}

\makeatletter
\renewcommand{\thesubsection}{%
  \ifnum\c@subsection<1 \@arabic\c@section
  \else \thesection.\@arabic\c@subsection
  \fi
}
\makeatother

\usepackage{eucal}
\usepackage{tikz-cd}

\DeclareMathOperator{\chr}{char}

\newcommand{\Frob}{\mathrm{Frob}}
\newcommand{\der}{\mathrm{der}}
\newcommand{\cen}{\mathrm{cen}}
\newcommand{\tor}{\mathrm{tor}}
\newcommand{\sconn}{\mathrm{sc}}

\title{Residual irreducibility of compatible systems}
\date{\today}

\author{Stefan Patrikis}
\address{Department of Mathematics, The University of Utah, Salt Lake City, UT}
\email{\href{mailto:patrikis@math.utah.edu}{patrikis@math.utah.edu}}
\urladdr{\url{http://www.math.utah.edu/~patrikis/}}

\author{Andrew Snowden}
\address{Department of Mathematics, University of Michigan, Ann Arbor, MI}
\email{\href{mailto:asnowden@umich.edu}{asnowden@umich.edu}}
\urladdr{\url{http://www-personal.umich.edu/~asnowden/}}

\author{Andrew Wiles}
\address{Mathematical Institute, University of Oxford, Oxford}
\email{\href{mailto:wiles@maths.ox.ac.uk}{wiles@maths.ox.ac.uk}}

\thanks{AS was supported by NSF grants DMS-1303082 and DMS-1453893}

\begin{document}

\begin{abstract}
We show that if $\{\rho_{\ell}\}$ is a compatible system of absolutely irreducible Galois representations of a number field then the residual representation $\ol{\rho}_{\ell}$ is absolutely irreducible for $\ell$ in a density~1 set of primes. The key technical result is the following theorem: the image of $\rho_{\ell}$ is an open subgroup of a hyperspecial maximal compact subgroup of its Zariski closure with bounded index (as $\ell$ varies). This result combines a theorem of Larsen on the semi-simple part of the image with an analogous result for the central torus that was recently proved by Barnet-Lamb, Gee, Geraghty, and Taylor, and for which we give a new proof.
\end{abstract}

\maketitle
\tableofcontents

\section{Introduction}

The purpose of this paper is to prove the following theorem:

\begin{theorem}[Informal version]
In a compatible system of absolutely irreducible Galois representations, a density~1 set of residual representations are absolutely irreducible.
\end{theorem}

In the rest of the introduction, we state this theorem precisely, indicate the main idea of the proof, and discuss the connection to some previous work.

\subsection{The main theorem}

Let $F$ and $E$ be number fields. Let $\Gamma_F$ be the absolute Galois group of $F$, let $\Sigma$ be a set of places of $E$, and let $\{\rho_v\}_{v \in \Sigma}$ be a collection where $\rho_v \colon \Gamma_F \to \GL_n(E_v)$ is a continuous representation (for some $n$ independent of $v$). We say that the family $\{\rho_v\}$ is a \emph{compatible system} if there exists a finite set $S$ of places of $F$ and a finite multi-set $I$ of integers such that the following conditions hold:
\begin{itemize}
\item Let $w \not\in S$ be a place of $F$ with residue characteristic different from that of $v \in \Sigma$. Then $\rho_v$ is unramified at $w$, and the characteristic polynomial of $\rho_v(\Frob_w)$ has coefficients in $E$ and is independent of $v$.
\item Let $w \not\in S$ be a place of $F$ with residue characteristic equal to that of $v \in \Sigma$. Then $\rho_v$ is crystalline at $w$ with Hodge--Tate weights in the multi-set $I$.
\end{itemize}
We remark that often in the definition of a compatible system, a stronger condition is required on ``labeled'' Hodge--Tate weights, a refinement of the multi-set $I$ into multi-sets indexed by embeddings $F \to \ol{E}$ (see \cite[\S 5.1]{blggt}); for us, however, the coarser condition just given suffices. For a set of rational prime numbers $\Pi$, let $\Sigma(\Pi)$ be the set of places of $E$ above a prime in $\Sigma$. We will typically work with compatible systems indexed by sets of this form. For a representation $\rho$ over a $p$-adic field, we write $\ol{\rho}$ for the residual representation (with respect to some lattice). Whenever we refer to the \emph{density} of a set of places of a number field, we mean the Dirichlet density. We can now state the precise version of our theorem:

\begin{theorem} \label{thm:irr}
Let $\{\rho_v\}_{v \in \Sigma}$ be a compatible system of representations of $\Gamma_F$ with coefficients in $E$, where $\Sigma=\Sigma(\Pi)$ for some set $\Pi$ of rational primes of density~1. Suppose that each $\rho_v$ is semi-simple, and let $\rho_v \otimes \ol{E}_v = \bigoplus_{i=1}^r \rho_{v,i}^{\oplus m_{v,i}}$ be its decomposition into irreducible representations over $\ol{E}_v$. (We assume $\rho_{v,i}$ and $\rho_{v,j}$ are non-isomorphic for $i \ne j$.) Then there exists a density~1 subset $\Pi' \subset \Pi$ such that for $v \in \Sigma'=\Sigma(\Pi')$ the representations $\ol{\rho}_{v,i}$ are irreducible and pairwise non-isomorphic (meaning $\ol{\rho}_{v,i}$ and $\ol{\rho}_{v,j}$ are non-isomorphic for $v \in \Sigma'$ and $i \ne j$).
\end{theorem}

\begin{corollary}\label{cor:irr}
Suppose each $\rho_v$ is absolutely irreducible. Then $\ol{\rho}_v$ is absolutely irreducible for all $v \in \Sigma'$.
\end{corollary}
%
In the special case of Hodge--Tate regular compatible systems (see \cite[\S 5.1]{blggt}), this corollary is due to Barnet-Lamb, Gee, Geraghty, and Taylor (see \cite[Proposition 5.3.2]{blggt}), who have used it to powerful effect in establishing potential automorphy theorems for certain compatible systems of Galois representations over CM fields. As in their application, but now in greater generality, we obtain the following corollary:

\begin{corollary}
Suppose each $\rho_v$ is absolutely irreducible. Let $v$ be a place in $\Sigma'$ whose residue characteristic is at least $2(n+1)$. Then $\rH^1(\ol{\rho}_v(\Gamma_F), \mathrm{Ad}^0(\ol{\rho}_v))=0$. 
\end{corollary}

\begin{proof}
Combine Corollary \ref{cor:irr} with \cite[Appendix, Theorem 9]{adequate}.
\end{proof}

In the special case of abelian varieties---an example far removed from the case of regular compatible systems---our theorem yields by Faltings' theorem:

\begin{corollary} \label{cor:av}
Let $A$ be an abelian variety over $F$ with $\End(A)=\bZ$. Then the representation of $\Gamma_F$ on $A[\ell]$ is absolutely irreducible for $\ell$ in a set of primes of density~1.
\end{corollary}

This corollary is a weaker version of a result of Zarhin (\cite[Corollary 5.4.5]{zarhin}) that identifies $\End(A[\ell])$ as $\End(A) \otimes_{\bZ} \bZ/\ell\bZ$ for all but finitely many primes $\ell$. The corollary can be generalized in an obvious way to allow $\End(A)$ to be an order in an arbitrary number field.

%

\subsection{Algebraic monodromy groups}

Let $G$ be an algebraic group. We write $G^{\circ}$ for the identity component of $G$, which we assume to be reductive. We write $G^{\der}$ for the derived subgroup of $G^{\circ}$, and $G^{\sconn}$ for the universal cover of $G^{\der}$, both of which are semi-simple. We let $G^{\tor}$ be the quotient $G^{\circ}/G^{\der}$, which is a torus. In a compatible system $\{\rho_v\}$, we write $G_v$ for the Zariski closure of the image of $\rho_v$. If $\rho_v$ is semi-simple then $G_v^{\circ}$ is reductive.

To prove Theorem~\ref{thm:irr}, it is enough to show that the image of $\rho_v$ in $G_v$ is sufficiently large. We think of $G_v$ as composed of three pieces: the component group $\pi_0(G_v)$, the semi-simple part $G_v^{\der}$, and toral part $G^{\tor}_v$. The image of $\rho_v$ surjects onto $\pi_0(G_v)$, and Serre showed that this group is independent of $\ell$. Larsen proved that the image of $\rho_v$ hits a big piece of $G_v^{\der}$. Our starting point is the observation that $\rho_v$ hits enough of the torus; this has been previously established in \cite[Proposition 5.2.2]{blggt}, but we hope that our argument is somewhat easier to follow. We have also included our argument because it contains a finiteness observation (see the summary after Corollary \ref{totalimage}) that may be of independent interest. Here is the technical result, proven in \S \ref{torus}:

\begin{theorem}[Compare \cite{blggt}, Proposition 5.2.2] \label{thm:torimg}
Let $\{\rho_{\ell}\}_{\ell \in \Sigma}$ be a compatible system of semi-simple representations of $\Gamma_F$ with coefficients in $\bQ$. Then there exists a positive integer $N$ with the following property: the projection of $\im(\rho_{\ell}) \cap G_{\ell}^{\circ}$ to $G^{\tor}_{\ell}$ has index at most $N$ in the maximal compact subgroup of $G^{\tor}_{\ell}(\bQ_{\ell})$, for all $\ell \in \Sigma$.
\end{theorem}

Combined with a theorem of Larsen, we obtain (see the discussion after the proof of Proposition~\ref{prop:hyper}):

\begin{theorem}\label{totalimage}
Let $\{\rho_{\ell}\}_{\ell \in \Sigma}$ be a compatible system of semi-simple representations of $\Gamma_F$ with coefficients in $\bQ$, where $\Sigma$ is a set of primes of density~1. Then there exists a positive integer $N$ and a subset $\Sigma' \subset \Sigma$, also of density~1, such that for $\ell \in \Sigma'$ we can find a (not necessarily connected) reductive group scheme $\cG_{\ell}/\bZ_{\ell}$ with generic fiber $G_{\ell}$ such that $\cG_{\ell}(\bZ_{\ell})$ contains $\im(\rho_{\ell})$ as an open subgroup of index at most $N$.
\end{theorem}

We now indicate the main idea in the proof Theorem~\ref{thm:torimg}. It suffices to treat the case where $G_{\ell}$ is connected for all $\ell$. Suppose that $\rho_{\ell}$ decomposes as a direct sum of irreducible representations $\rho_{\ell}^1, \ldots, \rho_{\ell}^{n(\ell)}$. The determinants of these summands control $G^{\tor}_{\ell}$. In particular, if one knew that these determinants formed compatible systems, then the theory of abelian Galois representations would show that their images are large enough to prove the theorem. Unfortunately, there seems to be no way at present to prove this compatibility.

Our main observation is that one can quite easily prove something very close to this compatibility, and which suffices for our purposes. Precisely, we notice that there are only finitely many possibilities for (the Hecke character associated to) $\det(\rho_{\ell}^i)$: indeed, its ramification is constrained, and there are only finitely many possibilities for its infinity type and its value on any particular Frobenius element. It follows that one can partition $\Sigma$ into finitely many subsets $\Sigma_1, \ldots, \Sigma_r$ such that $\{\det(\rho_{\ell}^i)\}_{\ell \in \Sigma_k}$ forms a compatible system for all $i$ and $k$.

\begin{remark}
Larsen's result is very general: it applies to compatible systems of representations of any profinite group possessing certain ``Frobenius elements.'' By contrast, Theorem~\ref{thm:torimg} (or at least our proof of it) is specific to Galois groups of number fields, and make critical (though simple) use of $\ell$-adic Hodge theory.
\end{remark}

\section{The image in the torus}\label{torus}

We now prove Theorem~\ref{thm:torimg}. Fix the $d$-dimensional compatible system $\{\rho_{\ell}\}_{\ell \in \Sigma}$. By Serre's theorem, after replacing $F$ with a finite Galois extension, we can assume that $G_{\ell}$ is connected for all $\ell$. Let $S$ be the set of places of ramification of the compatible system, and let $I$ be the finite multi-set of integers occurring as Hodge--Tate weights. The abelianization of the Galois group $\Gamma_{F,S}$ is topologically finitely generated. Let $T$ be a finite set of places of $F$ such that the Frobenius elements $\Frob_v$ with $v \in T$ topologically generate $\Gamma_{F,S}$. Denote by $\ol{\bQ}$ the algebraic closure of $\bQ$ in $\bC$. Fix once and for all an isomorphism $\ol{\bQ}_{\ell} \cong \bC$ for each $\ell$, and use these to regard $\ol{\bQ}$ as a subfield of $\ol{\bQ}_{\ell}$. For $v \in T$, let $\lambda_{v,1}, \cdots, \lambda_{v,d} \in \ol{\bQ} \subset \ol{\bQ}_{\ell}$ be the eigenvalues of $\rho_{\ell}(\Frob_v)$; notice that this set is independent of $\ell$ by compatibility. Let $\cH$ be the set of algebraic Hecke characters $\chi$ of $F$ satisfying the following conditions: (1) $\chi$ is unramified away from $S$; (2) the weights of $\chi$ belong to the set of finite sums of distinct elements of (the multi-set) $I$; and (3) $\chi(\Frob_v)$ is of the form $\lambda_{v,i_1} \cdots \lambda_{v,i_r}$ where the indices $i_1, \ldots, i_r$ are distinct. The set $\cH$ is finite.

 Via the fixed isomorphism $\ol{\bQ}_{\ell} \cong \bC$, we can associate to each Hecke character of $F$ a de~Rham character $\Gamma_F \to \ol{\bQ}_{\ell}^{\times}$, and vice versa. The following lemma exposes the key property of the set $\cH$:

\begin{lemma}
Let $\sigma$ be a summand of $\rho_{\ell} \otimes \ol{\bQ}_{\ell}$. Then the Hecke character corresponding to $\det(\sigma)$ belongs to $\cH$.
\end{lemma}

The Galois group $\Gamma_{\bQ}$ acts on $\cH$ by restricting its action on algebraic Hecke characters: for any algebraic Hecke character $\psi$, the finite part $\psi_f$ can be defined over $\ol{\bQ}$, and for any $\sigma \in \Gamma_{\bQ}$, there is then a unique (still algebraic) Hecke character ${}^\sigma \psi$ whose finite part is isomorphic to $\psi_f \otimes_{\ol{\bQ}, \sigma} \ol{\bQ}$. Let $\cS$ be a multi-set whose elements belong to $\cH$ and that is stable under $\Gamma_{\bQ_{\ell}}$. There is then an associated representation $\tau_{\cS,\ell} \colon \Gamma_F \to \GL_n(\bQ_{\ell})$, with $n=\# \cS$, such that $\tau_{\cS,\ell} \otimes \ol{\bQ}_{\ell}$ is the sum of the one-dimensional representations given by elements of $\cS$. Let $T_{\cS,\ell}/\bQ_{\ell}$ be the Zariski closure of the image of $\tau_{\cS,\ell}$.

\begin{proposition} \label{prop:bd}
There exists a positive integer $M$ such that the image of $\tau_{\cS,\ell}$ in $T_{\cS,\ell}(\bQ_{\ell})$ has index at most $M$ in the maximal compact, for any choice of $\ell$ and $\Gamma_{\bQ_{\ell}}$-stable $\cS$.
\end{proposition}

\begin{proof}
Let $T_F$ denote the restriction of scalars torus $\mathrm{Res}_{F/\bQ}(\bG_m)$, and let $T_{F, \ell}$ denote $T_F \otimes_{\bQ} \bQ_{\ell}$. The representation $\tau_{\cS, \ell}$ is Hodge--Tate and abelian, hence locally algebraic (\cite[A7 Theorem 3]{serre:l-adic}): there exists an algebraic morphism of tori $f_{\cS, \ell} \colon T_{F, \ell} \to T_{\cS, \ell}$ such that (leaving the class field theory identification implicit) $f_{\cS, \ell}|_{U}= \tau_{\cS, \ell}|_{U}$ for some open subgroup $U$ of $(F \otimes_{\bQ} \bQ_{\ell})^{\times}$. We claim that $f_{\cS, \ell}$ is surjective. Since $T_{\cS, \ell}(\bQ_{\ell})$ has a pro-$\ell$ open subgroup, and $\tau_{\cS, \ell}(U)$ contains an open pro-$\ell$ subgroup of $\tau_{\cS, \ell}(\Gamma_F)$, $\tau_{\cS, \ell}(U)$ is open in $\tau_{\cS, \ell}(\Gamma_F)$. As $T_{\cS, \ell}$ is connected, $\tau_{\cS, \ell}(U)=f_{\cS, \ell}(U)$ is Zariski-dense in $T_{\cS, \ell}$, so $f_{\cS, \ell}$ is a surjection.

Next, we are allowed to discard a finite number of $\ell$, so we now assume that the Hecke characters in $\cH$ are unramified at all primes above $\ell$, and that $F/\bQ$ is unramified at $\ell$. We may therefore assume $\tau_{\cS, \ell}$ is crystalline, and that the torus $T_{F, \ell}$, and therefore its quotient $T_{\cS, \ell}$, are unramified. These tori canonically extend to tori over $\bZ_{\ell}$, as does the morphism $f_{\cS, \ell}$ (all of which we continue to denote by the same symbols). Moreover, $\tau_{\cS, \ell}$ and $f_{\cS, \ell}$ in fact agree on all of $(\cO_F \otimes_{\bZ} \bZ_{\ell})^\times$. To see this, we can reduce to the case of Galois characters, where it is well-known (\cite[3.9.7 Corollary]{cco}).

To find an upper bound on the index $[T_{\cS, \ell}(\bZ_{\ell}):\tau_{\cS, \ell}(\Gamma_F)]$, it therefore suffices to bound the cokernel of the induced map on $\bZ_{\ell}$-points, i.e. the index $[T_{\cS, \ell}(\bZ_{\ell}):f_{\cS,\ell}(T_{F, \ell}(\bZ_{\ell})]$. We begin by bounding the order of the torsion subgroup of $X^{\bullet}(T_{F, \ell})/f_{\cS, \ell}^*(X^{\bullet}(T_{\cS, \ell}))$, where we write $X^\bullet$ for the character group. Denote by $a$ the degree $[F: \bQ]$. Each of the Hecke characters $\psi \in \cS \subset \cH$ has a collection of Hodge--Tate numbers, and there is an absolute bound, independent of $\ell$ and $\cS$, for the greatest of these integers; let us call it $N$. We can then identify the image $f_{\cS, \ell}^* X^\bullet(T_{\cS, \ell})$ with the submodule of $\bZ^a$ spanned by a collection of integer vectors, all of whose coordinates have absolute value at most $N$. Next note that for any morphism of abelian groups $f \colon \bZ^b \to \bZ^a$ for which a set of generators of $\bZ^b$ map to vectors whose coordinates are bounded by $N$, the order of $\mathrm{coker}(f)_{\tor}$ is bounded by a function of $a$ and $N$. Namely, $|\mathrm{coker}(f)_{\tor}|$ is equal to the greatest common divisor of the $r \times r$ minors, where $r$ equals the rank of $f$ (in our application, $r=b$). In general, then, the cokernel can have torsion order no greater than the largest determinant of an $a \times a$ integer matrix with entries bounded in absolute value by $N$.

We thus have a uniform bound for the torsion subgroup of $X^{\bullet}(T_{F, \ell})/f_{\cS, \ell}^*(X^{\bullet}(T_{\cS, \ell}))$. If we discard the finitely many $\ell$ below this bound, we see that $f_{\cS, \ell}$ is not only a surjection of tori, but is in fact a surjection of \'{e}tale sheaves on $\Spec \bZ_{\ell}$ (we can check this over an unramified extension of $\bZ_{\ell}$ that splits the tori, where by the invariant factor decomposition of the character groups it reduces to the assertion that multiplication by $n$ on $\bG_m$ is a surjection of \'{e}tale sheaves in characteristic not dividing $n$). We can therefore bound the index $[T_{\cS, \ell}(\bZ_{\ell}): f_{\cS, \ell}(T_{F, \ell}(\bZ_{\ell}))]$ by\footnote{The following calculation comes from \cite[Lemma A.1.6]{blggt}, although that lemma is stated somewhat imprecisely: multiplication by 2 on $\bf{G}_m$ over $\bZ_2$ is a counter-example, unless ``surjection" is read as ``surjection of \'{e}tale sheaves."} 
\begin{align*}
H^1_{\acute{e}t}(\Spec \bZ_{\ell}, \ker (f_{\cS, \ell}))&= H^1_{\acute{e}t}(\Spec \mathbf{F}_{\ell}, \ker(f_{\cS, \ell})) \\ &=H^1 \left(\Gamma_{\mathbf{F}_{\ell}}, \Hom(X^\bullet(T_{F, \ell})/f_{\cS, \ell}^* X^\bullet(T_{\cS, \ell}), \ol{\mathbf{F}}_{\ell}^\times) \right) \\&= \Hom(X^\bullet(T_{F, \ell})/f_{\cS, \ell}^* X^\bullet(T_{\cS, \ell}), \ol{\mathbf{F}}_{\ell}^\times)_{(\mathrm{Fr}_{\ell}-1)}.
\end{align*}
(The last subscript denotes $\mathrm{Fr}_{\ell}-1$ coinvariants.) Choose an integer $m$ so that $\Gamma_{\bF_{\ell^m}}$ acts trivially on $X^\bullet(T_{F, \ell})$. Then there is a surjection
\[
\Hom(X^\bullet(T_{F, \ell})/f_{\cS, \ell}^* X^\bullet(T_{\cS, \ell}), \ol{\mathbf{F}}_{\ell}^\times)_{(\mathrm{Fr}_{\ell^m}-1)} \to \Hom(X^\bullet(T_{F, \ell})/f_{\cS, \ell}^* X^\bullet(T_{\cS, \ell}), \ol{\mathbf{F}}_{\ell}^\times)_{(\mathrm{Fr}_{\ell}-1)},
\]
and the source of this map has order bounded by the order of the torsion subgroup of $X^\bullet(T_{F, \ell})/f_{\cS, \ell}^* X^\bullet(T_{\cS, \ell})$ (the torsion-free quotient has vanishing cohomology, by Hilbert 90), for which we have already obtained an independent of $\ell$ bound. The proposition follows.
\end{proof}

Fix a prime $\ell \in \Sigma$ for the moment. Decompose $\rho_{\ell} \otimes \ol{\bQ}_{\ell}$ as $\sigma_1 \oplus \cdots \oplus \sigma_r$ with $\sigma_i$ irreducible of dimension $d_i$. Over $\ol{\bQ}_{\ell}$, we can regard $G_{\ell}$ as a subgroup of $\GL_{d_1} \times \cdots \times \GL_{d_r} \subset \GL_d$. Consider the diagram 
\begin{displaymath}
\xymatrix{
G_{\ell} \ar[r] \ar[d] & \GL_{d_1} \times \cdots \times \GL_{d_r} \ar[d] \\
G^{\tor}_{\ell} \ar[r]^{\delta} & \bG_m^r }
\end{displaymath}
where the right vertical maps are determinants, and the map $\delta$ is induced by the universal property of $G^{\tor}$ (every map from $G$ to a torus factors through $G^{\tor}$). The central torus $Z$ of $G_{\ell}$ maps to the central $\bG_m$ in each $\GL_{d_i}$. The map $\bG_m^r \subset \GL_{d_1} \times \cdots \times \GL_{d_r} \stackrel{\det}{\to} \bG_m^r$ has kernel of size $d_1 \cdots d_r$. It follows that the kernel of the determinant map $Z \to \bG_m^r$ is also bounded by this, and so is the kernel of $\delta$. The composite $\Gamma_F \to G^{\tor}_{\ell}(\ol{\bQ}_{\ell}) \to \bG_m^r(\ol{\bQ}_{\ell})$ is given explicitly by $(\det(\sigma_1), \ldots, \det(\sigma_r))$. Let $\cS=\{\det(\sigma_i)\} \subset \cH$. This set is clearly $\Gamma_{\bQ_{\ell}}$-stable. Theorem~\ref{thm:torimg} now follows from Proposition~\ref{prop:bd} and the following lemma.

\begin{lemma}
There is an isogeny of tori $G^{\tor}_{\ell} \to T_{\cS,\ell}$ (defined over $\bQ_{\ell}$), the degree of which is bounded independent of $\ell$, such that the diagram
\begin{displaymath}
\xymatrix{
\Gamma_F \ar[r] \ar[rd]_{\tau_{\cS,\ell}} & G^{\tor}_{\ell}(\bQ_{\ell}) \ar[d] \\ & T_{\cS,\ell}(\bQ_{\ell}) }
\end{displaymath}
commutes.
\end{lemma}

\begin{proof}
The composite of the above homomorphism $\Gamma_F \to \bG_m^r(\ol{\bQ}_{\ell})$ with the direct sum map $\bG_m^r \to \GL_r$ is $\tau_{\cS, \ell} \otimes \ol{\bQ}_{\ell}$. By the universal property of $G^{\tor}_{\ell}$, there is a unique homomorphism $\pi_{\cS, \ell} \colon G^{\tor}_{\ell} \otimes \ol{\bQ}_{\ell} \to T_{\cS, \ell} \otimes \ol{\bQ}_{\ell}$ such that the diagram
\[
 \xymatrix{
\Gamma_F \ar[r] \ar[rd]_{\tau_{\cS,\ell}} & G^{\tor}_{\ell}(\ol{\bQ}_{\ell}) \ar[d]^{\pi_{\cS, \ell}} \\ & T_{\cS, \ell}(\ol{\bQ}_{\ell})
}
\]
commutes. Conjugating $\pi_{\cS, \ell}$ by any element of $\Gamma_{\bQ_{\ell}}$ yields another such diagram, since the image of $\Gamma_F$ in both $G_{\ell}^{\tor}$ and $T_{\cS, \ell}$ lies in the $\bQ_{\ell}$-points. By uniqueness, we conclude that $\pi_{\cS, \ell}$ is $\Gamma_{\bQ_{\ell}}$-invariant, hence descends to a morphism $\pi_{\cS, \ell} \colon G^{\tor}_{\ell} \to T_{\cS, \ell}$. It is a surjection since $T_{\cS, \ell}$ is by definition the Zariski-closure of the image of $\tau_{\cS, \ell}$; and its kernel is finite, bounded independently of $\ell$ by $d_1 \cdots d_r$.
\end{proof}
\section{Irreducibility of residual representations}\label{grouptheory}

\subsection{Extending automorphisms of group schemes}

\begin{lemma}
Let $\cO$ be a complete DVR with fraction field $K$ and residue field $k$, let $X/\cO$ be a smooth affine scheme, and let $f \colon X_K \to X_K$ be a morphism of $K$-schemes such that $f$ carries $X(\cO_L)$ into itself for every finite unramified extension $L/K$. Then $f$ extends uniquely to a morphism of $\cO$-schemes $X \to X$.
\end{lemma}

\begin{proof}
Write $X=\Spec(A)$, with $A$ a flat $\cO$-algebra of finite type, so that $f$ corresponds to a map of $K$-algebras $f^* \colon A_K \to A_K$. We must show that $f^*$ carries $A$ into itself. Let $\varphi$ be an element of $A$, and let $\pi$ be a uniformizer of $\cO$. Let $n \ge 0$ be minimal so that $\pi^n f^*(\varphi) \in A$. We claim $n=0$, which will prove the lemma. Suppose not. Let $x$ be an $\ell$-point of $X$ for some finite extension $\ell/k$, and let $L$ be the unramified extension of $K$ corresponding to $\ell$. Since $X$ is smooth and $\cO_L$ is complete, $x$ lifts to a $\cO_L$-point $\wt{x}$ of $X$. We have $f^*(\varphi)(\wt{x})=\varphi(f(\wt{x})) \in \cO_L$ since $f(\wt{x}) \in X(\cO_L)$. Since $n>0$, it follows that $(\pi^n f^*(\varphi))(x)=0$. Thus $\pi^n f^*(\varphi)$ vanishes on all $\ol{k}$-points of the special fiber, and thus vanishes on the special fiber since the special fiber is reduced. That is, $\pi^n f^*(\varphi)=0$ in $A/\pi A$, and so $\pi^n f^*(\varphi) \in \pi A$. But this implies $\pi^{n-1} f^*(\varphi) \in A$ since $A$ is flat over $\cO$, a contradiction.
\end{proof}

\begin{lemma} \label{extaut}
Let $K/\bQ_{\ell}$ be a finite extension, let $\cO=\cO_K$, let $\cG/\cO$ be a simply connected semi-simple group, and let $f \colon \cG_K \to \cG_K$ be an automorphism of the generic fiber such that $f$ carries $\cG(\cO)$ into itself. Then $f$ extends uniquely to an automorphism $\cG \to \cG$ of $\cO$-groups.
\end{lemma}

\begin{proof}
By \cite[Lemma~6.2]{bigness}, $f$ carries $\cG(\cO_L)$ into itself for any tamely ramified finite extension $L/K$. Therefore, by the previous lemma, $f$ uniquely extends to a map $\cG \to \cG$ of $\cO$-schemes, which is necessarily a group automorphism by uniqueness.
\end{proof}

%
%

\begin{remark}
In fact, Lemma~\ref{extaut} is true when $\cG$ is connected reductive. One can deduce the general result from the simply connected case fairly easily.
\end{remark} 

\subsection{A result on hyperspecial groups}
\label{ss:hyper}

Let $K/\bQ_{\ell}$ be a finite extension and let $\cO=\cO_K$ be the ring of integers. Let $G/K$ be a reductive group with $G^{\circ}$ unramified and let $\Gamma \subset G(K)$ be a Zariski dense profinite group. We use the following notation:
\begin{displaymath}
\xymatrix{
G^{\sconn} \ar[r]^{\pi} \ar[rd]_{\sigma} & G^{\circ} \ar[d]^{\tau} \\ & G^{\ad} }
\end{displaymath}
(Note that $\pi$ factors as $G^{\sconn} \to G^{\der} \subset G^{\circ}$.) We let $\Gamma^{\circ} = \Gamma \cap G^{\circ}(K)$.

\begin{proposition} \label{prop:hyper}
Assume $\sigma^{-1}(\tau(\Gamma^{\circ}))$ is a hyperspecial subgroup of $G^{\sconn}(K)$ and that $\Gamma^{\tor}$ is open in $G^{\tor}(K)$. Also assume that the residue characteristic $\ell$ is sufficiently large compared to $\dim(G)$. Then there exists a reductive group $\cG/\cO$ with generic fiber $G$ such that $\Gamma$ is an open subgroup of $\cG(\cO)$. Moreover, the index of $\Gamma$ in $\cG(\cO)$ can be bounded in terms of the dimension of $G$ and the index of $\Gamma^{\tor}$ in the maximal compact of $G^{\tor}(K)$.
\end{proposition}

\begin{proof}
Let $\cG^{\sconn}/\cO$ be the given semi-simple group with generic fiber $G^{\sconn}$ such that $\cG^{\sconn}(\cO)=\sigma^{-1}(\tau(\Gamma^{\circ}))$. Let $\cG^{\circ}$ be the unique extension of $G^{\circ}$ to $\cO$ for which the map $G^{\sconn} \to G^{\circ}$ extends. One can construct this extension as follows. Let $T$ be the central torus in $G^{\circ}$, and let $\cT/\cO$ be the unique torus with generic fiber $T$ (note that $T$ is unramified by our assumption that $G^{\circ}$ is unramified). Let $\cZ/\cO$ be the center of $\cG^{\sconn}$, which is finite \'etale (due to our assumption on $\ell$), and let $Z=\cZ_K$. Let $H$ be the kernel of the map $T \times G^{\sconn} \to G^{\circ}$, which is contained in $T \times Z$. The order $n$ of $H$ can be bounded in terms of $\dim(G)$, and is therefore small compared to $\ell$. Let $\cH$ be the closure of $H$ in $\cT \times \cZ$. Since $H$ is contained in the finite \'etale group scheme $\cT[n] \times \cZ$ it follows that $\cH$ is finite \'etale. The group $\cG^{\circ}$ is then the quotient of $\cT \times \cG^{\sconn}$ by $\cH$, which exists by general theory. We do not prove uniqueness of $\cG^{\circ}$, as it is not needed in the proof.

Let $\gamma$ be an element of $\Gamma$. Then conjugation by $\gamma$ induces an automorphism of $G^{\circ}$ which descends to $G^{\ad}$ and then lifts (uniquely) to $G^{\sconn}$. Denote this automorphism of $G^{\sconn}$ by $f'_{\gamma}$. Let $g \in \cG^{\sconn}(\cO)$. Then $\sigma(g)=\tau(\alpha)$ for some $\alpha \in \Gamma^{\circ}$. We have
\begin{displaymath}
\sigma(f'_{\gamma}(g))=\gamma \sigma(g) \gamma^{-1}=\gamma \tau(\alpha) \gamma^{-1}=\tau(\beta)
\end{displaymath}
where $\beta=\gamma \alpha \gamma^{-1}$. Since $\beta \in \Gamma^{\circ}$, it follows that $f'_{\gamma}(g) \in \sigma^{-1}(\tau(\Gamma^{\circ}))=\cG^{\sconn}(\cO)$. Thus $f'_{\gamma}$ maps $\cG^{\sconn}(\cO)$ to itself, and therefore, by Lemma~\ref{extaut}, extends uniquely to an automorphism of $\cG^{\sconn}$, still denoted $f'_{\gamma}$.

It is clear that conjugation by $\gamma$ carries $T$ into itself. It therefore extends uniquely to an automorphism $f''_{\gamma}$ of $\cT$, as $\cT$ is the N\'eron model of $T$.

Now consider the quotient map $\phi \colon \cT \times \cG^{\sconn} \to \cG^{\circ}$. Let $t \in \cT(\ol{K})$ and $h \in \cG^{\sconn}(\ol{K})$. Then
\begin{displaymath}
\phi(f''_{\gamma}(t), f'_{\gamma}(h))=f''_{\gamma}(t) \pi(f'_{\gamma}(h))=(\gamma t \gamma^{-1}) (\gamma \pi(h) \gamma^{-1}) = \gamma \phi(t, h) \gamma^{-1}.
\end{displaymath}
It follows that $(f''_{\gamma}, f'_{\gamma})$ maps $\cH_{\ol{K}}$, and therefore $\cH$, to itself, and thus descends to an automorphism $f_{\gamma}$ of $\cG^{\circ}$. Moreover, it follows that $f_{\gamma}$ induces conjugation by $\gamma$ on the generic fiber of $\cG^{\circ}$.

We thus see that $\Gamma$ normalizes $\cG^{\circ}(\cO) \subset G(K)$. In particular, $\Gamma^{\circ}$ is a compact subgroup of $G^{\circ}(K)$ normalizing $\cG^{\circ}(\cO)$, and must therefore be contained in $\cG^{\circ}(\cO)$, as $\cG^{\circ}(\cO)$ is maximal compact.

We now define $\cG$ to be the pushout of $\Gamma^{\circ} \to \Gamma$ along the map $\Gamma^{\circ} \to \cG^{\circ}$. Explicitly, write $\Gamma = \coprod_{i=1}^r \gamma_i \Gamma^{\circ}$ with $\gamma_i \in \Gamma$. Then $\cG = \coprod_{i=1}^r \gamma_i \cG^{\circ}$. Suppose $\gamma_i \gamma_j = \gamma_k \gamma'$ with $\gamma' \in \Gamma^{\circ}$. Then we define a map
\begin{displaymath}
\gamma_i \cG^{\circ} \times \gamma_j \cG^{\circ} \to \gamma_k \cG^{\circ}
\end{displaymath}
by mapping $(\gamma_i g, \gamma_j h)$ to $\gamma_k \gamma' f_{\gamma_j^{-1}}(g) h$. Note that $(g, h) \mapsto \gamma' f_{\gamma_j^{-1}}(g) h$ is a well-defined map of $\cO$-schemes $\cG^{\circ} \times \cG^{\circ} \to \cG^{\circ}$. Combining these maps for all $i$ and $j$ gives the multiplication on $\cG$.

It is clear that $\cG$ is a reductive group over $\cO$ with generic fiber $G$, and that $\Gamma$ is contained in $\cG(\cO)$. We now show that $\Gamma$ is open. It suffices to show $\Gamma^{\circ}$ is open. Since $\sigma^{-1}(\tau(\Gamma^{\circ}))$ is open in $G^{\sconn}(K)$, it follows that $\tau(\Gamma^{\circ})$ is open. The image of $\Gamma^{\circ}$ in $\tau(\Gamma^{\circ}) \times \Gamma^{\tor}$ is open, and thus the image of $\Gamma^{\circ}$ under the map $G^{\circ}(K) \to G^{\ad}(K) \times G^{\tor}(K)$ is open. It follows that $\Gamma^{\circ}$ is open in $G^{\circ}(K)$.

Finally, we bound the index of $\Gamma$ in $\cG(\cO)$. Since $\Gamma$ meets every connected component, it suffices to bound the index of $\Gamma^{\circ}$ in $\cG^{\circ}(\cO)$. The kernel of the map $\phi \colon \cG^{\circ}(\cO) \to \cG^{\ad}(\cO) \times \cG^{\tor}(\cO)$ is finite, with order bounded in terms of $\dim(G)$, so it suffices to bound the index of $\phi(\Gamma^{\circ})$ in the target. The group $\phi(\Gamma^{\circ})$ is contained in $\tau(\Gamma^{\circ}) \times \Gamma^{\tor}$, and the index can be bounded in terms of $\dim(G)$. It thus suffices to bound the index of $\tau(\Gamma^{\circ})$ in $\cG^{\ad}(\cO)$. We have $\tau(\Gamma^{\circ})=\sigma(\cG^{\sconn}(\cO))$, so we must control the index of the image of the map $\cG^{\sconn} \to \cG^{\ad}$ on $\cO$ points. Since this map is a surjection of group schemes with finite \'etale kernel $\cZ$, it follows that the index is bounded by the order of $\rH^1_{\rm et}(\Spec(\cO), \cZ)$. This can be bounded in terms of the order of $\cZ$, and thus in terms of $\dim(G)$.
\end{proof}
Theorem~\ref{totalimage} now follows by combining Proposition~\ref{prop:hyper}, Theorem~\ref{thm:torimg}, and Larsen's theorem (\cite[Thm.~3.17]{larsen}) on hyperspecial image.

\subsection{Some representation theory}

Let $k$ be a field, let $G/k$ be a connected semi-simple group, and let $V$ be an algebraic representation of $G_{\ol{k}}$. Pick a maximal torus $T$ in $G_{\ol{k}}$. For a weight $\lambda$ of $T$, we define $\Vert \lambda \Vert$ to be the maximum of $\vert \langle \lambda, \alpha^{\vee} \rangle \vert$ as $\alpha$ varies over the roots of $G_{\ol{k}}$. We define $\Vert V \Vert$ to be the maximum of $\Vert \lambda \Vert$ over those weights appearing in $V$. This is independent of the choice of $T$. If $k$ has characteristic~0 then (for fixed $\dim(G)$) one can bound $\dim(V)$ in terms of $\Vert V \Vert$, and vice versa.

Now suppose that $G/k$ is a torus and $\lambda$ is a character of $G_{\ol{k}}$. Define $\Vert \lambda \Vert$ to be $\# \pi_0(\ker(\lambda))$. For a representation $V$ of $G_{\ol{k}}$ define $\Vert V \Vert$ to be the maximum of $\Vert \lambda \Vert$ over characters $\lambda$ appearing in $V$.

Finally, suppose that $G/k$ is a reductive group and $V$ is a representation of $G_{\ol{k}}$. We define $\Vert V \Vert_{\der}$ to be the norm of $V$ restricted to $G^{\der}$ and $\Vert V \Vert_{\cen}$ to be the norm of $V$ restricted to the central torus of $G^{\circ}$.

\begin{proposition} \label{prop:inv1}
Let $\cO$ be a complete DVR with fraction field $K$ and residue field $k$, let $\cG/\cO$ be a reductive group, and let $\cV$ be a representation of $\cG$ over $\cO$. Suppose that $\dim(\cG)$, $\dim(\cV)$, and $\# \pi_0(G)$ are small compared to $\chr(k)$. Then the natural map $(\cV^{\cG})_k \to \cV_k^{\cG_k}$ is an isomorphism.
\end{proposition}

%

\begin{proof}
The map in question is injective, and we must prove it is surjective. That is, given a $\cG_k$-invariant vector $v_1$ in $\cV_k$, we must produce a lift to a $\cG$-invariant vector in $\cV$. Suppose that $v_n$ is a lift of $v_1$ to a $\cG$-invariant vector in $\cV/\pi^n \cV$, where $\pi$ is a uniformizer of $\cO$. Let $X$ be the fiber of the reduction map $\cV/\pi^{n+1} \cV \to \cV/\pi^n \cV$ over $v_n$. Since $v_n$ is invariant, $X$ is stable for the action of $\cG$, and is naturally a torsor for $\cV_k$. Thus $X$ is classified by an element of $\rH^1(\cG_k, \cV_k)$. (This cohomology is algebraic group cohomology. Explicitly, let $\wt{v}$ be a lift of $v_n$. Then $g \mapsto g\wt{v}-\wt{v} \in \pi^n \cV/\pi^{n+1} \cV \cong \cV_k$ is a representating cocycle.) We claim that this cohomology group vanishes. First, the representation $\cV_k$ of $\cG_k^{\circ}$ has small norm (see, for instance, \cite[Prop.~3.5]{bigness}). An element of $\rH^1(\cG^{\circ}_k, \cV_k)$ is represented by an extension of $\cV_k$ by the trivial representation. However, such an extension still has small norm, and is therefore semi-simple (see, for instance, \cite[Prop.~3.3]{bigness}). Thus $\rH^1(\cG^{\circ}_k, \cV_k)$ vanishes. The map
\begin{displaymath}
\rH^1(\cG_k, \cV_k) \to \rH^1(\cG^{\circ}_k, \cV_k)
\end{displaymath}
is injective, since $\# \pi_0(G)$ is coprime to $\chr(k)$, and so $\rH^1(\cG_k, \cV_k)$ vanishes. Thus $X$ is a trivial torsor, and therefore has a fixed point $v_{n+1}$. We thereby obtain a compatible sequence $(v_n)_{n \ge 0}$ of $\cG$-invariant elements which defines the requisite $\cG$-invariant element in $\cV$.
\end{proof}

\begin{proposition} \label{prop:inv2}
Let $k$ be a finite field, let $G/k$ be a reductive group, and let $\Gamma \subset G(k)$ be a subgroup of small index (compared to $\chr(k)$) that meets every geometric component. Let $V$ be a representation of $G$ over $\ol{k}$. Assume that $\dim(G)$, $\Vert V \Vert_{\der}$, and $\Vert V \Vert_{\cen}$ are small compared to $\chr(k)$. Then $V^G=V^{\Gamma}$.
\end{proposition}

\begin{proof}
It suffices to treat the case where $G$ is connected. And we can treat separately the cases where $G$ is a torus and $G$ is semi-simple. We first treat the case where $G$ is semi-simple. It suffices to treat the case where $G$ is simply connected. But note then that $G(k)$ has no small index subgroups, and so $\Gamma=G(k)$. The result then follows from the fact that $V^G=V^{G(k)}$, which relies on $\dim(G)$ and $\Vert V \Vert_{\der}$ being small compared to $\chr(k)$ (see, for instance, \cite[Prop.~3.6]{bigness}).

We now treat the case where $G$ is a torus. It suffices to treat the case where $V$ is one-dimensional. Let $\lambda$ be the character giving the action of $G$. It suffices to treat the case where $\lambda$ is non-trivial, and show that $\lambda \vert_{\Gamma}$ is non-trivial. Note that $\lambda$ factors as $G_{\ol{k}} \stackrel{p}{\to} \bG_m \stackrel{n}{\to} \bG_m$ where $p$ is the projection onto a direct factor and $\vert n \vert = \Vert \lambda \Vert$ is small by the assumption on $\Vert V \Vert_{\cen}$. Since $p(\Gamma)$ has small index in $\bG_m(\ol{k})$, so does $p(\Gamma)^n$, and so $\lambda$ is non-trivial on $\Gamma$.
\end{proof}

\subsection{A general irreducibility result}

We return to the setting of \S \ref{ss:hyper}.

\begin{theorem} \label{genirred}
Let $G/K$ be a reductive group with $G^{\circ}$ unramified, and let $\Gamma \subset G(K)$ be a Zariski dense profinite group such that $\sigma^{-1}(\tau(\Gamma^{\circ}))$ is hyperspecial. Assume that $\Gamma^{\tor}$ is open in $G^{\tor}(K)$ and its index in the maximal compact is small compared to $\chr(k)$. Also assume that $\dim(G)$ and $\# \pi_0(G)$ are small compared to $\chr(k)$.

Let $(\rho, V)$ be a representation of $G$ over $\ol{K}$ with $\dim(V)$ small compared to $\chr(k)$. Then there exists a $\Gamma$-stable lattice $\cV$ in $V$ such that the natural map
\begin{displaymath}
\End_{\Gamma}(\cV) \otimes \ol{k} \to \End_{\Gamma}(\cV_{\ol{k}})
\end{displaymath}
is an isomorphism.
\end{theorem}

\begin{corollary}
Suppose that $\rho$ decomposes as $\bigoplus_{i=1}^r \rho_i^{\oplus m_i}$ where the $\rho_i$ are pairwise non-isomorphic irreducible representations. Then the $\ol{\rho}_i$ are irreducible and pairwise non-isomorphic and (for an appropriate choice of lattice) $\ol{\rho}$ decomposes as $\bigoplus_{i=1}^r \ol{\rho}_i^{\oplus m_i}$.
\end{corollary}


\begin{proof}[Proof of theorem]
By Proposition~\ref{prop:hyper}, we can find a reductive group $\cG/\cO$ with generic fiber $G$ such that $\Gamma$ is an open subgroup of $\cG(\cO)$ with small index. Descend $V$ to a representation $V'$ over a finite extension $K'/K$, and let $\cV$ be a $\cG$-stable lattice in $V'$ (see \cite[\S 1.12]{larsen} for a proof of existence). Let $\cE=\End(\cV)$, which is also a representation of small norm (again by \cite[Prop.~3.5]{bigness}). We have
\begin{displaymath}
(\cE^{\Gamma})_{\ol{k}} = (\cE^{\cG})_{\ol{k}} = (\cE_{\ol{k}})^{\cG_{\ol{k}}} = (\cE_{\ol{k}})^{\Gamma}
\end{displaymath}
The first equality follows from the fact that $\Gamma$ and $\cG$ have the same Zariski closure. The second is Proposition~\ref{prop:inv1} (working over $\cO^{\rm un}_{K'}$), while the third is Proposition~\ref{prop:inv2}.
\end{proof}

\subsection{Proof of the main theorem}
Let $\{\rho_v\}_{v \in \Sigma}$ be a $d$-dimensional compatible system of representations of $\Gamma_F$ with coefficients in $E$, where $\Sigma=\Sigma(\Pi)$. For $\ell \in \Pi$ define $\rho_{\ell} = \bigoplus_{v \mid \ell} \rho_v$. We regard $\rho_{\ell}$ as a representation of $\Gamma_F$ over the field $\bQ_{\ell}$ of dimension $d \cdot [E:\bQ_{\ell}]$. As such, $\{\rho_{\ell}\}_{\ell \in \Pi}$ forms a compatible system with coefficients in $\bQ$. Let $\Pi' \subset \Pi$ be a density one subset for which Larsen's theorem on hyperspecial image holds, and fix $\ell \in \Pi'$. Let $G$ be the Zariski closure of the image of $\rho_{\ell}$. Then we can regard $\rho_v$ as an algebraic representation of $G$ over the field $E_v$. The result now follows from Theorem~\ref{genirred}, which we may apply since by Larsen's theorem $G^\circ$ is unramified, by Theorem~\ref{thm:torimg} the image of $\Gamma^{\tor}$ in the maximal compact of $G^{\tor}(\bQ_{\ell})$ is bounded independently of $\chr(k)$, and by Serre's theorem the component group $\pi_0(G_{\ell})$ is independent of $\ell$.
\qed

\subsection{Further results}

In \cite{blggt}, a somewhat stronger result is needed, and proven, for application to potential automorphy theorems for Hodge--Tate regular compatible systems. We can prove the general version of this as well:
\begin{corollary}
Let $\{\rho_v\}_{v \in \Sigma}$ be a compatible system of representations of $\Gamma_F$ with coefficients in a number field $E$, where $\Sigma= \Sigma(\Pi)$ for some set $\Pi$ of rational primes of density 1. Suppose that each $\rho_v$ is absolutely irreducible. Then there exists a density 1 subset $\Pi' \subset \Pi$ such that for all $v \in \Sigma'= \Sigma(\Pi')$, lying above $\ell \in \Pi'$, the restriction $\ol{\rho}_v|_{\Gamma_F(\zeta_{\ell})}$ is absolutely irreducible.
\end{corollary}
\proof
This time we consider the compatible system of $\Gamma_F$-representations with $\bQ$-coefficients $\rho_{\ell}'= \{\kappa_{\ell} \oplus \bigoplus_{v \mid \ell} \rho_{v}\}_{\ell}$, where $\kappa_{\ell}$ denotes the $\ell$-adic cyclotomic character. By Frobenius reciprocity, we must show that $\Hom_{\Gamma_F}(\ol{\rho}_v, \ol{\rho}_v(i))=0$ for all $i =1, \ldots, \ell-2$, for all $v$ lying above some density one set of rational primes $\ell$. Again let $\Pi' \subset \Pi$ be a density one subset for which Larsen's theorem on hyperspecial image holds, fix $\ell \in \Pi'$, and let $G' \subset G \times \bG_m$ be the Zariski closure of $\rho_{\ell}'$, with $G$ continuing to denote the Zariski closure of $\rho_{\ell}= \bigoplus_{v \mid \ell} \rho_v$. Both $\kappa_{\ell}$ and $\rho_v$ are algebraic representations of $G'$ over $E_v$, and so $\rho_v(i)$ is as well, for any integer $i$. We then follow the proof of Theorem \ref{genirred}, but without the hypothesis that $\Gamma^{', \tor}$ has `small' index in the maximal compact subgroup of $G^{', \tor}(\bQ_{\ell})$ ($\Gamma'$ here is taken to be $\rho'_{\ell}(\Gamma_F)$). Proposition~\ref{prop:hyper} produces a reductive group $\cG'/\bZ_{\ell}$ with generic fiber $G'$ such that $\Gamma'$ is an open subgroup of $\cG'(\bZ_{\ell})$. Let $\cV$ be a $\cG'$-stable lattice in $\rho_v$, and consider the representation $\cE_i= \Hom(\cV, \cV(i))$. Exactly as in the proof of Theorem \ref{genirred}, we have
\[
(\cE_i^{\Gamma'})_{\ol{k}}= (\cE_i^{\cG'})_{\ol{k}}= (\cE_{i, \ol{k}})^{\cG'_{\ol{k}}}.
\]

We now discard from $\Pi'$ the finitely many rational primes such that $[F(\zeta_{\ell}):F]$ is smaller than $\ell -1$, and claim that we can also deduce that $(\cE_{i, \ol{k}})^{\cG'_{\ol{k}}}= (\cE_{i, \ol{k}})^{\Gamma'}$ for $i=1, \ldots, \ell-2$. Indeed, all the hypotheses of Proposition \ref{prop:inv2} are still satisfied, with the possible exception that $\Vert \cE_i \Vert_{\cen}$ may not be small compared to $\chr(k)$. Following the proof of Proposition~\ref{prop:inv2}, however, we need only check that upon restricting $\cE_{i, \ol{k}}$ to the maximal central torus $Z$ of $\cG'_{\ol{k}}$, any irreducible (one-dimensional) constituent on which $Z$ acts non-trivially is also acted on non-trivially by the image of $\Gamma'$ in $Z$. Since $Z$ acts trivially on $\ol{\rho}_v^* \otimes \ol{\rho}_v$ ($\ol{\rho}_v$ being irreducible, by the main theorem), we are reduced to the assertion that $\ol{\kappa}^i$ is a non-trivial $\Gamma_F$-representation for $i=1, \ldots, \ell-2$, which holds by our restriction on $\ell$.

The corollary now results from the fact that $\Hom_{\Gamma_F}(\rho_v, \rho_v(i))=0$ for any $i \neq 0$.
\endproof
We likewise deduce the following application:
\begin{corollary}
Let $\{\rho_v\}_{v \in \Sigma}$ be a compatible system of representations with coefficients in $E$, where $\Sigma=\Sigma(\Pi)$ for $\Pi$ a set of rational primes of density 1. Suppose that each $\rho_v$ is absolutely Lie-irreducible, in the sense that for all finite extensions $L/F$, the restriction $\rho_v|_{\Gamma_L}$ remains absolutely irreducible. Then for any integer $d$, there is a density one subset $\Pi_d \subset \Pi$ such that for all $v \in \Sigma(\Pi_d)$, and all extensions $L/F$ of degree at most $d$, the restriction $\ol{\rho}_v|_{\Gamma_L}$ is absolutely irreducible. 
\end{corollary}

\proof
As above, let $\{\rho_{\ell}= \bigoplus_{v \mid \ell} \rho_v \}_{\ell \in \Pi}$ be the associated compatible system with coefficients in $\bQ$. By Corollary~\ref{totalimage}, there is a positive integer $N$, a density one subset $\Pi' \subset \Pi$, and for each $\ell \in \Pi'$ a reductive group scheme $\cG_{\ell}/\bZ_{\ell}$ with generic fiber $G_{\ell}$ such that $\cG_{\ell}(\bZ_{\ell})$ contains $\Gamma= \rho_{\ell}(\Gamma_F)$ as an open subgroup of index at most $N$. For any $L/F$ satisfying $[L:F] \leq d$, it follows that $\cG_{\ell}(\bZ_{\ell})$ contains $\rho_{\ell}(\Gamma_L)$ with index at most $Nd$, and that it contains $\rho_{\ell}(\Gamma_L) \cap \Gamma^0$ with index at most $Nd \pi_0$, where $\pi_0$ is the order of the component group (independent of $\ell$ by Serre's theorem). Consider the representation $\rho_v$ over $E_v$ of the group $G_{\ell}$, and as before let $\cV$ be a $\cG_{\ell}$-stable lattice, and let $\cE= \End(\cV)$. Finally let $\ol{k}_v$ be an algebraic closure of the residue field of $E_v$. Provided $\ell$ is sufficiently large, we then have
\[
(\cE^{\rho_{\ell}(\Gamma_L) \cap \Gamma^0})_{\ol{k}_v}= (\cE^{\cG^0_{\ell}})_{\ol{k}_v}= (\cE_{\ol{k}_v})^{\cG^0_{\ell, \ol{k}_v}}= (\cE_{\ol{k}_v})^{\rho_{\ell}(\Gamma_L) \cap \Gamma^0}.
\]
The first equality holds because $\cV$ is absolutely Lie-irreducible, so the invariants are in both cases just the scalars. The second is again Proposition~\ref{prop:inv1}, and the third is again Proposition~\ref{prop:inv2}, invoking the uniform index bound.
\endproof

\end{document}